\documentclass[a4paper]{amsart}

\usepackage{color}
\usepackage[dvipdfm]{hyperref}
\usepackage{amssymb}
\usepackage{amsthm}
\usepackage{amscd}
\usepackage{amsmath}
\usepackage[all]{xy}
\usepackage{graphicx}
\theoremstyle{plain}
\newtheorem{theorem}{Theorem}
\newtheorem*{theorem*}{Theorem}

\newtheorem*{maintheorem*}{Main Theorem}
\newtheorem{proposition}{Proposition}
\newtheorem{corollary}[proposition]{Corollary}
\newtheorem{lemma}[proposition]{Lemma}

\newtheorem*{conjecture*}{Conjecture}

\theoremstyle{definition}
\newtheorem{definition}{Definition}
\newtheorem*{definition*}{Definition}
\newtheorem{example}{Example}
\newtheorem*{example*}{Example}
\newtheorem*{notation*}{Notation}
\newtheorem*{notation-conv*}{Notation and convention}
\newtheorem*{convention*}{Convention}
\newtheorem*{hypothesis*}{Hypothesis}

\theoremstyle{remark}
\newtheorem{remark}{Remark}

\newcommand{\ie}{\emph{i.e.}}
\newcommand{\mf}[1]{\mathfrak{#1}}

\newcommand{\myleq}{\leqslant}
\newcommand{\mygeq}{\geqslant}

\newcommand{\ZZ}{{\mathbb Z}}
\newcommand{\CC}{{\mathbb C}}
\newcommand{\Q}{{\mathbb Q}}
\newcommand{\IR}{{\mathbb R}}

\newcommand{\FF}{{\mathbb F}}

\newcommand{\GL}{{\rm GL}}

\newcommand{\bm}[1]{\mbox{\boldmath $#1$}} 

\newcommand{\im}{\mathop{\mathrm{im}}\nolimits}

\newcommand{\PolyTors}[2]{\Delta_{#1}^{#2}}

\newcommand{\Tor}[3]{\mathrm{Tor} ({#1}, {#2}, {#3})}

\newcommand{\tensor}[2]{#1 \otimes #2}

\newcommand{\cbasis}[2][c]{\mathbf{#1}^{#2}}

\newcommand{\bbasis}[2][b]{\mathbf{#1}^{#2}}
 
\newcommand{\hbasis}[2][h]{\mathbf{#1}^{#2}}

\newcommand{\bord}{\partial}

\newcommand{\ucover}[1]{\widetilde{#1}}
\newcommand{\abcover}[1]{\hat{#1}}

\newcommand{\hatK}{\hat{K}}
\newcommand{\hatvarphi}{\widehat{\varphi}}
\newcommand{\hatrho}{\widehat{\rho}}

\begin{document}


\title{
The twisted Alexander polynomial for 
finite abelian covers over three manifolds with boundary
}


\author{J\'er\^ome Dubois \and Yoshikazu Yamaguchi}

\address{Institut de Math\'ematiques de Jussieu, 
Universit\'e Paris Diderot--Paris 7, 
UFR de Math\'ematiques, 
Case 7012, B\^atiment Chevaleret,
75205 Paris Cedex 13
France}
\email{dubois@math.jussieu.fr}

\address{Department of Mathematics, 
Tokyo Institute of Technology, 
2-12-1 Ookayama Meguro-ku, Tokyo 152-8551, Japan}
\email{shouji@math.titech.ac.jp}

\date{\today}

\begin{abstract}
  We provide the twisted Alexander polynomials of finite abelian
  covers over three--dimensional manifolds whose boundary is a finite
  union of tori.  This is a generalization of a well--known formula
  for the usual Alexander polynomial of knots in finite cyclic
  branched covers over the three--dimensional sphere.
\end{abstract}

\maketitle



\section{Introduction}
The classical Alexander polynomial is defined for null--homologous
knots in rational homology spheres, where null--homologous means that
the homology class of a knot is trivial in the first homology group
with $\ZZ$-coefficients of the ambient space.

If the pair $(\abcover{M}, \hatK)$ of a rational homology sphere
$\abcover{M}$ and a null-homologous knot $\hatK$ in $\abcover{M}$ is given
by a finite cyclic branched cover over $S^3$ branched along a knot
$K$, where $\hatK$ is the lift of $K$, then we can compute the
Alexander polynomial of $\hatK$ by using the well--known formula
\[
\Delta_{\hatK}(t) =
\prod_{\xi \in \{x \in \CC\,|\, x^k = 1 \}} \Delta_K(\xi t)
\quad \hbox{up to a factor $\pm t^{a} \,  (a \in \ZZ)$}
\]
where $k$ is the order of the covering transformation group, 
$\xi$ runs all over the $k$-th roots of unity and
$\Delta_K(t)$ is the Alexander polynomial of $K$.
Such formulas have been investigated from the viewpoint of 
Reidemeister torsion for a long time. In particular, V.~Turaev gave a formula for the Alexander polynomial of $\hatK$ in a finite cyclic branched cover over $S^3$,
and a generalization in the case of links in general three--dimensional manifolds
(we refer to~\cite[Theorems~1.9.2 and~1.9.3 ]{Turaev86:torsion_in_knot_theory}). 

The purpose of this paper is to provide the generalization of the
above formula giving the Alexander polynomial of a knot in a finite
cyclic branched cover over $S^3$ to a formula for the twisted
Alexander polynomial of finite abelian covers, which is a special
kind of Reidemeister torsion.  Especially, we also consider the
twisted Alexander polynomial for a link in a three--dimensional manifold from the
viewpoint of Reidemeister torsion in the same way as V.~Turaev.  But
to deal with finite abelian covers beyond finite cyclic covers, we
adopt the approach of J.~Porti in his work~\cite{Porti:2004}. Porti
gave a new proof of Mayberry--Murasugi's formula, which gives the order of the first homology group of
finite abelian branched covers over $S^3$ branched along links, by using
Reidemeister torsion theory.
We call the twisted Alexander polynomial {\it the polynomial torsion} regarded as a kind of Reidemeister torsion.

In this paper, we are interested in the Reidemeister torsion for a
finite sheeted abelian covering. We are mainly intested in link
exteriors in homology three--spheres and their abelian covers.  Our main
theorem (see Theorem~\ref{theorem:torsion_covering}) is stated for an abelian cover
$\abcover{M} \to M$ between two three--dimensional manifolds whose boundary is a finite union of tori as
follows
\[
    \PolyTors{\abcover{M}}{\tensor{\hatvarphi}{\hatrho}}(\bm{t})
    = \epsilon \cdot \prod_{\xi \in \hat{G}} \PolyTors{M}{(\tensor{\varphi}{\rho}) \otimes \xi}(\bm{t})
\]
where $\PolyTors{\abcover{M}}{\tensor{\hatvarphi}{\hatrho}}(\bm{t})$
and $\PolyTors{M}{(\tensor{\varphi}{\rho}) \otimes \xi}(\bm{t})$ are
the signed twisted Alexander polynomials, $\hat{G}$ is the set of
homomorphisms from the covering transformation group $G$ to the
non--zero complex numbers and 
$\epsilon$ is a sign determined by the homology orientations of $\abcover{M}$ and $M$.

To be more precise, we need two homomorphisms of the fundamental group
to define the twisted Alexander polynomial of a manifold.  The symbol
$\varphi$ denotes a surjective homomorphism from $\pi_1(M)$ to a
multiplicative group $\ZZ^n$ and $\rho$ denotes a representation of
$\pi_1(M)$, \ie, a homomorphism from $\pi_1(M)$ to a linear
automorphism group $Aut(V)$ of some vector space $V$ (see
Section~\ref{section:def_polynomial_torsion} for the definition of the
polynomial torsion). In the definition of the twisted Alexander
polynomial of $\abcover{M}$, we use the pull--backs $\hatvarphi$ and
$\hatrho$ of $\varphi$ and $\rho$ to $\pi_1(\abcover{M})$.  The
homomorphisms $\varphi$ and $\xi$ determine variables in the twisted
Alexander polynomial of $\abcover{M}$. In our main theorem, we assume that the composition of $\xi$ with the
quotient homomorphism $\pi_1(M) \to \pi_1(M) / \pi_1(\abcover{M})
\simeq G$ factors through homomorphism $\varphi$ (see
Section~\ref{section:Fox}).

When we choose $\abcover{M} \to M$ as a finite cyclic cover of a knot
exterior $E_K$ of $K$ in $S^3$, $\varphi$ is the abelianization homomorphism
$\pi_1(E_K) \to \pi_1(E_K)/[\pi_1(E_K), \pi_1(E_K)] \simeq \ZZ$ and
$\rho$ is the one--dimensional trivial representation, 
Theorem~\ref{theorem:torsion_covering} reduces to the classical formula for the Alexander
polynomial of $\hatK$, where $\hatK$ is the lift of the knot in the finite
cyclic branched cover over $S^3$
$$
\Delta_{\hatK}(t) =  \prod_{\xi \in \{x \in \CC \,|\, x^k =1\}} \Delta_K(\xi t),
$$
up to a factor $\pm t^a$ $(a \in \ZZ)$, where $k$ is the order of
$\pi_1(M)/ \pi_1(\abcover{M})$.  Our formula also provides the
Alexander polynomial of a link in finite abelian branched covers over
$S^3$ branched along the link.  
 
\section*{Organization}

The outline of the paper is as
follows. Section~\ref{section:preliminaries} deals with some reviews
on the sign--determined Reidemeister torsion for a manifold. In
Section~\ref{section:def_polynomial_torsion}, we give the definition
of the polynomial torsion (the twisted Alexander polynomial) for a
manifold whose boundary is a finite union of tori.  In
Section~\ref{section:Fox}, we consider the polynomial torsion of
finite abelian covering spaces (see
Theorem~\ref{theorem:torsion_covering}).

\section*{Acknowledgments}

The authors gratefully acknowledge the many helpful ideas and
suggestions of J.~Porti during the preparation of the paper.  They
also wish to express their thanks to C.~Blanchet and S.~Friedl for
several helpful comments and for their encouragement.  
Our thanks also go to T.~Morifuji, K.~Murasugi and V.~Maillot
for giving us helpful informations 
and comments regarding this work.
The paper was conceived when J.D. visited
the CRM. He thanks the CRM for hospitality.  The first author (J.D.) is partially supported by the French ANR project ANR-08-JCJC-0114-01.
The second author (Y.Y.) is partially supported by the GCOE program at Graduate School of
Mathematical Sciences, University of Tokyo and Research Fellowships of Japan Society for the promotion of Science for Young Scientists. Y.Y. visited CRM and IMJ
while writing the paper. He thanks CRM and IMJ for their hospitality. The authors also would like to thank the referee for his comments and helpful remarks.

\section{Preliminaries}
\label{section:preliminaries}

\subsection{The Reidemeister torsion}
We review the basic notions and results about the sign--determined
Reidemeister torsion introduced by V. Turaev which are needed in this
paper. Details can be found in Milnor's survey~\cite{Milnor:1966} and
in Turaev's monograph~\cite{Turaev:2002}.

\subsubsection*{Torsion of a chain complex}

Let 
$C_* = ( 
  0 \to
    C_n \xrightarrow{d_n} 
    C_{n-1} \xrightarrow{d_{n-1}} 
    \cdots \xrightarrow{d_1} 
    C_0 \to 
  0 
)$ 
be a
chain complex of finite dimensional vector spaces over a field
$\FF$. Choose a basis $\cbasis{(i)}$ of $C_i$ and a basis $\hbasis{i}$
of the $i$-th homology group $H_i(C_*)$. The torsion of $C_*$ with
respect to these choices of bases is defined as follows.

For each $i$, let $\bbasis{i}$ be a set of vectors in $C_{i}$ such that
$d_{i}(\bbasis{i})$ is a basis of 
$B_{i-1}= \im(d_{i} \colon C_{i} \to C_{i-1})$ and 
let $\hbasis[\tilde{h}]{i}$ denote a lift of
$\hbasis{i}$ in $Z_i = \ker(d_{i} \colon C_i \to C_{i-1})$. The set of
vectors $d_{i+1}(\bbasis{i+1})\hbasis[\tilde{h}]{i}\bbasis{i}$ is a
basis of $C_i$. Let
$[d_{i+1}(\bbasis{i+1})\hbasis[\tilde{h}]{i}\bbasis{i}/\cbasis{i}] \in \FF^*$ 
denote the determinant of the transition matrix between those
bases (the entries of this matrix are coordinates of vectors in
$d_{i+1}(\bbasis{i+1})\hbasis[\tilde{h}]{i}\bbasis{i}$ with respect to
$\cbasis{i}$). The \emph{sign-determined Reidemeister torsion} of
$C_*$ (with respect to the bases $\cbasis{*}$ and $\hbasis{*}$) is the
following alternating product (see~\cite[Definition~3.1]{Turaev:2000}):
\begin{equation}\label{def:RTorsion}
  \Tor{C_*}{\cbasis{*}}{\hbasis{*}} 
  = (-1)^{|C_*|} \cdot  
  \prod_{i=0}^n 
  [d_{i+1}(\bbasis{i+1})\hbasis[\tilde{h}]{i}\bbasis{i}/\cbasis{i}]^{(-1)^{i+1}} 
  \in \FF^*.
\end{equation}
Here $$|C_*| = \sum_{k\mygeq 0} \alpha_k(C_*) \beta_k(C_*),$$ where
$\alpha_i(C_*) = \sum_{k=0}^i \dim C_k$ and 
$\beta_i(C_*) = \sum_{k=0}^i \dim H_k(C_*)$.

The torsion $\Tor{C_*}{\cbasis{*}}{\hbasis{*}}$ does not depend on the
choices of $\bbasis{i}$ nor on the lifts $\hbasis[\tilde{h}]{i}$.  Note that if
$C_*$ is acyclic (\ie~if $H_i = 0$ for all $i$), then $|C_*| = 0$.

\subsubsection*{Torsion of a CW-complex}

Let $W$ be a finite CW-complex and $(V, \rho)$ be a pair of a vector
space with an inner product over $\FF$ and a homomorphism of $\pi_1(W)$
into $Aut(V)$.  The vector space $V$ turns into a right $\ZZ[\pi_1(W)]$-module denoted $V_{\rho}$ by using the right
action of $\pi_1(W)$ on $V$ given by $v \cdot \gamma = \rho(\gamma)^{-1}(v)$, for
$v \in V$ and $\gamma \in \pi_1(W)$.  The complex of the universal cover 
with integer coefficients $C_*(\ucover{W}; \ZZ)$ also inherits a left
$\ZZ[\pi_1(W)]$-module structure via the action of $\pi_1(W)$ on
$\ucover{W}$ as the covering group.
We define the $V_\rho$-twisted
chain complex of $W$ to be
\[
C_*(W; V_\rho) = V_{\rho} \otimes_{\ZZ[\pi_1(W)]} C_*(\ucover{W}; \ZZ).
\]
The complex $C_*(W; V_\rho)$ computes the {\it $V_\rho$-twisted homology} of $W$
which is denoted by $H_*(W;V_\rho)$.

Let $\left\{e^{i}_1, \ldots, e^{i}_{n_i}\right\}$ be the set of
$i$-dimensional cells of $W$. We lift them to the universal cover and
we choose an arbitrary order and an arbitrary orientation for the
cells $\left\{ {\tilde{e}^{i}_1, \ldots, \tilde{e}^{i}_{n_i}}
\right\}$. If we choose an orthonormal basis $\{\bm{v}_1, \ldots,
\bm{v}_m\}$ of $V$, then we consider the corresponding basis
$$
\cbasis{i} = 
  \left\{ 
    \bm{v}_1 \otimes \tilde{e}^{i}_{1} , 
    \ldots, 
    \bm{v}_m \otimes \tilde{e}^{i}_{1}, 
    \cdots, 
    \bm{v}_1 \otimes \tilde{e}^{i}_{n_i}, 
    \ldots, 
    \bm{v}_m \otimes \tilde{e}^{i}_{n_i}
  \right\}
$$ 
of 
$C_i(W; V_\rho) = V_\rho \otimes_{\ZZ[\pi_1(W)]} C_*(\ucover{W};\ZZ)$. 
We call the basis $\cbasis{*} = \oplus_{i} \cbasis{i}$ 
{\it a geometric basis} of $C_*(W;V_\rho)$.
Now choosing for each $i$ a basis $\hbasis{i}$ of the
$V_\rho$-twisted homology $H_i(W; V_\rho)$, we can compute the torsion
$$\Tor{C_*(W; V_\rho)}{\cbasis{*}}{\hbasis{*}} \in \FF^*.$$

We mainly consider the torsion of acyclic chain complexes $C_*(W;V_\rho)$, \ie,
the homology group $H_*(W; V_\rho) = \bm{0}$. For acyclic chain complex $C_*(W;V_\rho)$, 
this definition only depends on the combinatorial class of $W$, the
conjugacy class of $\rho$, the choices of $\cbasis{*}$.
The basis $\cbasis{*}$ for $C_*(W;V_\rho)$ depends on the following choices:
\begin{itemize}
\item[1.] an order of cells $\{e^i_j\}$ and an orientation of each $\{e^i_j\}$;
\item[2.] a lift $\tilde{e}^{i}_{j}$ of $e^i_j$ and;
\item[3.] an orthonormal basis of the vector space $V$.
\end{itemize}
We summarize the effect of changing these choices to $\Tor{C_*(W; V_\rho)}{\cbasis{*}}{\emptyset}$ in the following three remarks.

\begin{remark}
  We have the same $\Tor{C_*(W; V_\rho)}{\cbasis{*}}{\emptyset}$ for
  all orthonormal bases of $V$ since the effect of change of
  orthonormal bases in $V$ is given by multiplying the determinant of
  the bases change matrix with power of $\chi(W)$.  If the Euler characteristic
  $\chi(W)$ is zero, then we have the same torsion for any basis of $V$.
\end{remark}

\begin{remark}
  The torsion $\Tor{C_*(W; V_\rho)}{\cbasis{*}}{\emptyset}$ depends on the choice of the
  lifts $\tilde{e}^{i}_j$ under the action of $\pi_1(W)$ by $\rho$.
  The effect of different lift of a cell is expressed as the determinant
  of $\rho(\gamma)$ for some $\gamma$ in $\pi_1(W)$.  To avoid this
  problem, we often use representations into ${\mathrm{SL}(V)}$.
\end{remark}

\begin{remark}\label{rem:sign_refinement}
  To define the Reidemeister torsion, we order the cells $\{e^i_j\}$ and chose an orientation of each $e^i_j$, if we choose a different order and different orientations of cells, we could change the torsion sign.  To remove this sign ambiguity, that only occurs when $m$ is odd, we use the fact that the sign of the torsions $\Tor{C_*(W;
    \IR)}{\cbasis{*}_\IR}{\hbasis{*}_\IR}$ and $\Tor{C_*(W;
    \IR)}{\cbasis{*}}{\hbasis{*}}$ change in the same way.
\end{remark}

Therefore we usually consider the torsion $\Tor{C_*(W; V_\rho)}{\cbasis{*}}{\emptyset}$
up to the above indeterminacy, namely up to a factor 
$\pm \det \rho(\gamma)$ for some $\gamma$ in $\pi_1(W)$.

We can construct the additional sign term referred to in Remark~\ref{rem:sign_refinement} 
as follows.
The cells 
$\left\{ \left. 
  \tilde{e}^{i}_j \, 
         \right|\, 
  0 \myleq i \myleq \dim W, 
  1 \myleq j \myleq n_i
\right\}$ are in one--to--one correspondence with the cells of $W$, 
their order and orientation are induced an order and 
an orientation for the cells 
$\left\{ \left. 
  e^{i}_j \,
         \right|\, 
  0 \myleq i \myleq \dim W, 
  1 \myleq j \myleq n_i
\right\}$. 
Again, corresponding to these
choices, we get a basis $\cbasis{i}_\IR$ over $\IR$ of $C_i(W; \IR)$.

Choose a \emph{homology orientation} of $W$, which is an orientation
of the real vector space $H_*(W; \IR) = \bigoplus_{i\mygeq 0}
H_i(W; \IR)$. Let $\mf{o}$ denote this chosen
orientation. Provide each vector space $H_i(W; \IR)$ with a reference
basis $\hbasis{i}_\IR$ such that the basis 
$\left\{ {\hbasis{0}_\IR, \ldots, \hbasis{\dim W}_\IR} \right\}$ of 
$H_*(W; \IR)$ is {positively oriented} with respect to $\mf{o}$. 
Compute the sign--determined Reidemeister torsion 
$\Tor{C_*(W; \IR)}{\cbasis{*}_\IR}{\hbasis{*}_\IR} \in \IR^*$ of 
the resulting based and homology based chain complex and consider its sign 
$$\tau_0 =
  \mathrm{sgn}\left(
    \Tor{C_*(W; \IR)}{\cbasis{*}_\IR}{\hbasis{*}_\IR}
   \right)
\in \{\pm 1\}.$$

We define the \emph{sign--refined twisted Reidemeister torsion} of $W$
{(with respect to $\mf{o}$)} to be
\begin{equation}\label{eqn:TorsionRaff}
  \tau_0^m \cdot 
  \Tor{C_*(W; V_\rho)}{\cbasis{*}}{\emptyset} \in \FF^* 
\end{equation}
where $m = \dim_{\FF} V$. 
This sign refinement also works for the chain complex $C_*(W;V_\rho)$ with non--trivial homology group.
When the dimension of $V$ is even, we do not need the sign refinement, 
\ie,
the torsion $\Tor{C_*(W; V_\rho)}{\cbasis{*}}{\emptyset}$ is determined 
up to $\det \rho(\gamma)$ for some $\gamma$ in $\pi_1(W)$.

One can prove that the sign--refined Reidemeister torsion is invariant 
under cellular subdivision, 
homeomorphism and simple homotopy equivalences. In fact,
it is precisely the sign $(-1)^{|C_*|}$ in Equation~(\ref{def:RTorsion})
which ensures all these important invariance properties to hold (see~\cite{Turaev:2002}).

\section{Definition of the polynomial torsion}
\label{section:def_polynomial_torsion}

In this section we define the polynomial torsion. This gives a point
of view from the Reidemeister torsion to polynomial invariants of
topological space.

Hereafter $M$ denotes a compact and connected three--dimensional
manifold such that its boundary $\partial M$ is empty or a disjoint
union of $b$ two--dimensional tori: $$\partial M = T^2_1 \cup \ldots \cup T^2_b.$$

In the sequel, we denote by $V$ a vector space over $\CC$ and by
$\rho$ a representation of $\pi_1(M)$ into $Aut(V)$, and such that
$\det \rho(\gamma) = 1$ for all $\gamma \in \pi_1(M)$.

Next we introduce a twisted chain complex with some variables.
It will be done by using a $\ZZ[\pi_1(M)]$--module with variables 
to define a new twisted chain complex.
We regard $\ZZ^n$ as the multiplicative group generated by $n$ variables
$t_1, \ldots, t_n$, \ie, 
$$\ZZ^n = \big\langle t_1, \ldots, t_n \,|\, t_it_j = t_jt_i \, (\forall i, j)\big\rangle$$
and 
consider  a surjective homomorphism $\varphi\colon \pi_1(M)\to\ZZ^n$.
We often abbreviate the $n$ variables $(t_1, \ldots, t_n)$ to $\bm{t}$ and 
the rational functions $\CC(t_1, \ldots, t_n)$ to $\CC(\bm{t})$.

When we consider the right action of $\pi_1(M)$ on $V(\bm{t}) = \CC(\bm{t}) \otimes V$
by the tensor representation 
\[
\varphi \otimes \rho^{-1}\colon \pi_1(M) \to Aut(V(\bm{t})),
\quad \gamma \mapsto \varphi(\gamma) \otimes \rho^{-1}(\gamma), 
\]
we have the associated twisted chain $C_*(M; V_\rho(\bm{t}))$ given by
$$
C_*(M; V_\rho(\bm{t})) = V_\rho(\bm{t}) \otimes_{\ZZ[\pi_1(M)]} C_*(\ucover{M}; \ZZ)
$$
where $ f \otimes v \otimes \gamma \cdot \sigma$ is identified with 
$f \varphi(\gamma) \otimes {\rho(\gamma)^{-1}}(v) \otimes \sigma$ for
any 
$\gamma \in \pi_1(M)$, 
$\sigma \in C_*(\ucover{M};\ZZ)$, 
$v \in V$ and 
$f \in \CC(\bm{t})$.
We call this complex {\it the $V_\rho (\bm{t})$-twisted chain complex} of $M$.

\begin{definition}
Fix a homology orientation on $M$. 
  If $C_*(M;V_\rho(\bm{t}))$ is acyclic, then the sign--refined Reidemeister torsion
  of $C_*(M;V_\rho(\bm{t}))$
  $$
  \PolyTors{M}{\tensor{\varphi}{\rho}}(t_1, \ldots, t_n) = 
    \tau_0^m \cdot
    \Tor{C_*(M;V_\rho(\bm{t}))}{\cbasis{*}}{\emptyset} 
    \in \CC(t_1, \ldots, t_n) \setminus \{0\}
  $$
  is called the \emph{polynomial torsion} of $M$.
\end{definition}

Observe that the sign--refined Reidemeister torsion
$\PolyTors{M}{\tensor{\varphi}{\rho}}(t_1, \ldots, t_n)$ is determined up to a factor
$t_1^{m_1} \cdots t_n^{m_n}$ like the classical Alexander polynomial.

\begin{example}[J.~Milnor~\cite{Milnor:1968}, P.~Kirk \& C. Livingston~\cite{KL}]
  \label{example:torsion_one_dim}
  Suppose that $M$ is the knot exterior $E_K = S^3 \setminus N(K)$ of
  a knot $K$ in $S^3$ where $N(K)$ is an open tubular neighbourhood of $K$. 
  
  If the representation $\rho \in \mathrm{Hom}(\pi_1(E_K); \Q)$ is the
  trivial homomorphism and $\varphi$ is the abelianization of
  $\pi_1(E_K)$, \ie, $\varphi \colon \pi_1(E_K) \to H_1(E_K;\ZZ) \simeq
  \langle t\rangle$, then the twisted chain complex $C_*(E_K;
  \Q(t)_\rho)$ is acyclic and the Reidemeister torsion
  $\Delta_{E_K}^{\varphi \otimes \rho}(t)$ is expressed as a rational
  function which is the Alexander polynomial $\Delta_K(t)$ divided by
  $(t-1)$ (see also~\cite{Turaev:2000, Turaev:2002}).
\end{example}

\begin{example}
  \label{example:torsion_link_one_dim}
  Suppose now that $M$ is the link exterior $E_L = S^3 \setminus N(L)$
  of a link $L$ in $S^3$. We suppose that $L$ has $n$ components, where $n \geqslant 2$.
  We denote by $\mu_i$ the meridian of the $i$-th component.
  Consider the abelianization $\varphi \colon \pi_1(E_L)
  \to \ZZ^n$ defined by $\varphi(\mu_i) = t_i$. Let $\rho \colon
  \pi_1(E_L) \to \GL(1;\CC) = \CC \setminus \{0\}$ be the
  one--dimensional representation such that $\rho(\mu_i) = \xi_i$.
  Then the twisted chain complex
  $C_*(E_L; \CC(\bm{t})_\rho)$ is acyclic and the Reidemeister torsion
  $\Delta_{E_L}^{\varphi \otimes \rho}(t_1, \ldots, t_n)$ is given by
  (up to $\pm (\xi_1^{-1} t_1)^{k_1} \cdots (\xi_n^{-1} t_n)^{k_n}$, $k_i \in \ZZ$)
  \[
  \Delta_{E_L}^{\varphi \otimes \rho}(t_1, \ldots, t_n)
  = \Delta_L(\xi_1^{-1} t_1, \ldots, \xi_n^{-1} t_n)
  \] 
  where $\Delta_L(t_1, \ldots, t_n)$
  is the Alexander polynomial of $L$.
\end{example}

\section{Torsion for finite sheeted abelian coverings}
\label{section:Fox}

\subsection{Statement of the result}
Let $\abcover{M}$ be a finite sheeted abelian covering of $M$, where $M$ denotes a compact and connected three--dimensional
manifold such that its boundary $\partial M$ is empty or a disjoint
union of $b$ two--dimensional tori: $$\partial M = T^2_1 \cup \ldots \cup T^2_b.$$
We denote by $p$  the induced homomorphism from $\pi_1(\abcover{M})$ to $\pi_1(M)$ by 
the covering map $\abcover{M} \to M$.  
The associated deck transformation group is a finite abelian group $G$ of order $|G|$. 
We endow the manifolds $M$ and $\abcover{M}$ with some arbitrary homology orientations.

We have the following exact sequence:
\begin{equation}\label{seq:covering_induced_hom}
  1 \to
  \pi_1(\abcover{M}) \xrightarrow{p}
  \pi_1(M) \xrightarrow{\pi}  
  G  \to 1.
\end{equation}

When we consider the polynomial torsion for $\abcover{M}$, we use the
pull--back of homomorphisms of $\pi_1(M)$ as homomorphisms of
$\pi_1(\abcover{M})$.  We denote by $\varphi$ a surjective
homomorphism from $\pi_1(M)$ to $\ZZ^n$ and by $\hatvarphi$ the
pull-back by $p$.  We also suppose that $\pi$ factors through $\varphi$.
Our situation is summarized as follows:
$$
\xymatrix@R15pt{
  \pi_1(\abcover{M}) \ar[r]^{p_*} \ar[dr]_{\hatvarphi} & \pi_1(M) \ar@{->>}[d]^\varphi \ar[r]^{\pi}& G \\
                       &  \ZZ^n \ar@{-->}[ur].
}
$$

Similarly we use the symbol $\hatrho$ for the pull--back of $\rho \colon \pi_1(M) \to Aut(V)$
by $p$, where $V$ is a vector space.
For homomorphisms of the quotient group $G \simeq \pi_1(M) / \pi_1(\abcover{M})$, we use 
the Pontrjagin dual of $G$ which is the set of all representations
$\xi \colon G \to \CC^* = \CC \setminus \{0\}$ from $G$ to non--zero complex numbers. 
Let $\hat{G}$ denote this space.

We give the statement of
the polynomial torsion for abelian coverings via that of the based manifold.
\begin{theorem}\label{theorem:torsion_covering}
  With the above notation, we suppose that the twisted chain complex $C_*(M; V_\rho(\bm{t}))$ is acyclic.
  Then the twisted chain complex $C_*(\abcover{M}; V_{\hatrho}(\bm{t}))$ is also acyclic and the polynomial torsion
  is expressed as
  \begin{equation}\label{eqn:covering_formula}
    \PolyTors{\abcover{M}}{\tensor{\hatvarphi}{\hatrho}}(\bm{t})
    = \epsilon \cdot \prod_{\xi \in \hat{G}} \PolyTors{M}{(\tensor{\varphi}{\rho}) \otimes \xi}(\bm{t})
  \end{equation}
  where $\epsilon$ is a sign equal to 
  $\tau_0(\abcover{M})^m \cdot \tau_0(M)^{m|G|}$ and $m=\dim V$.
\end{theorem}

\begin{remark}
As we already observed, the sign term in Equation~(\ref{eqn:covering_formula}) is not relevant when $m$ is even.
\end{remark}

\begin{remark}[Explanation of Formula~(\ref{eqn:covering_formula}) with variables]
\label{remark:variables_in_formula}
If we denote by $\bar{\xi}$ the composition $\xi \circ \bar{\pi}$ as in the following commutative diagram 
$$ 
\xymatrix@R15pt{
  \pi_1(M) \ar[r]^{\pi} \ar@{->>}[d]_{\varphi} & G \ar[r]^{\xi} & \CC \\
  \ZZ^n \ar[ur]_{\bar{\pi}}
}
$$
then Formula~(\ref{eqn:covering_formula}) can be written concretely as follows:
$$
\PolyTors{\abcover{M}}{\tensor{\hatvarphi}{\hatrho}} 
(t_1, \ldots, t_n) =
  \epsilon \cdot
  \prod_{\xi \in \hat{G}}
    \PolyTors{M}{\tensor{\varphi}{\rho}} (t_1 \bar{\xi}(t_1), \ldots, t_n\bar{\xi}(t_n))
$$
\end{remark}

In the special case where $n=1$, $G = \ZZ/q\ZZ$ and $\abcover{M}$ is
the $q$--fold cyclic covering $M_q$ of $M$, then we have that $\bar{\xi}(t)
= e^{2\pi k \sqrt{-1} /q}$, for $k = 0, \ldots, q-1$.  Hence we have the following covering formula for the polynomial torsion.
\begin{corollary}\label{cor:cyclic_twistedAlexnader}
  Suppose that $\varphi(\pi_1(M)) = \langle t \rangle$ and
  $\hatvarphi(\pi_1(M_q)) = \langle s \rangle \subset \langle t \rangle$,
  where we suppose that $s=t^q$. We have
  \[
    \PolyTors{M_q}{\tensor{\hatvarphi}{\hatrho}}(s) =
    \PolyTors{M_q}{\tensor{\hatvarphi}{\rho}}(t^q) = \epsilon \cdot 
    \prod_{k=0}^{q-1} \PolyTors{M}{\tensor{\varphi}{\rho}}(e^{2\pi k \sqrt{-1} /q}t).
  \]
\end{corollary}

The torsion $\PolyTors{M_q}{\tensor{\hatvarphi}{\rho}}(t^q)$ in
Corollary~\ref{cor:cyclic_twistedAlexnader} can be regarded as a kind
of the total twisted Alexander polynomial introduced
in~\cite{SilverWilliams:DynamicsTwistedAlexander}.  Hirasawa and
Murasugi~\cite{HirasawaMurasugi:hakone2007} worked on the total
twisted Alexander polynomial for abelian representations as in
Example~\ref{example:torsion_one_dim} and
they observed the similar formula as in 
Corollary~\ref{cor:cyclic_twistedAlexnader} in terms of the total
Alexander polynomial and the Alexander polynomial of a knot in the
cyclic branched coverings over $S^3$.

\subsection{Proof of Theorem~\ref{theorem:torsion_covering}}

We use the same notation as in Remark~\ref{remark:variables_in_formula}.

First observe the following key facts:
\begin{itemize}
\item the universal cover $\ucover{M}$ of $M$ is also the one of
  $\abcover{M}$,
\item the torsion $\PolyTors{\abcover{M}}{\tensor{\hatvarphi}{\rho}}$ is
  computed using the twisted complex $$V_\rho({t_1, \ldots, t_n})
  \otimes_{\ZZ[\pi_1(\abcover{M})]} C_*(\ucover{M}; \ZZ),$$
\item whereas the torsion $\PolyTors{M}{\tensor{\varphi}{\rho}}$ is
  computed using $$V_\rho({t_1, \ldots, t_n}) \otimes_{\ZZ[\pi_1({M})]}
  C_*(\ucover{M}; \ZZ).$$
\end{itemize}
 
   \begin{lemma}
 Let $x \in V_\rho(t_1, \ldots, t_n)$, $c \in C_*(\ucover{M}; \ZZ)$. For $\gamma \in \pi_1(M)$, $x\gamma^{-1} \otimes_{\pi_1(\abcover{M})} \gamma c$ only depends on $\pi(\gamma) \in G$. For $g \in G$, choose $\gamma \in \pi_1(M)$ such that $\pi(\gamma) = g$ and set 
    \begin{equation}\label{eqn:G-action_bar_M}
      g \star (x \otimes_{\pi_1(\abcover{M})} c)
      = x\gamma^{-1} \otimes_{\pi_1(\abcover{M})} \gamma c.
    \end{equation}
   This defines a natural action of $G$ on $V_\rho(t_1, \ldots, t_n) \otimes_{\pi_1(\abcover{M})} C_*(\ucover{M}; \ZZ)$.
   \end{lemma}

  Further observe that, since for any lift $\gamma$ of $g$, $\gamma$ is not contained in
  $p(\pi_1(\abcover{M}))$, we can not reduce the right hand side in
  Equation~$(\ref{eqn:G-action_bar_M})$.

\begin{proof}
  Take another lift $\gamma'$ in $\pi_1(M)$ of $g \in G$. Since $\gamma' = \hat\gamma \gamma $
  for some $\hat\gamma \in p(\pi_1(\abcover{M}))$, we can see that 
  $x\gamma'^{-1} \otimes_{\pi_1(\abcover{M})} \gamma' c
    = x\gamma^{-1} \otimes_{\pi_1(\abcover{M})} \gamma c.$
\end{proof}

The proof of Theorem~\ref{theorem:torsion_covering} is based on the following
technical lemma.
\begin{lemma}\label{lemma:Fox}
  The map
  \begin{equation}
    \Phi \colon 
      V_\rho({\bm{t}}) \otimes_{\ZZ[\pi_1(\abcover{M})]} C_*(\ucover{M}; \ZZ) \to 
      (V_\rho({\bm{t}}) \otimes_\CC \CC[G]) \otimes_{\ZZ[\pi_1({M})]} C_*(\ucover{M}; \ZZ)
  \end{equation}
  given by
   $$\Phi(x \otimes_{\pi_1(\abcover{M})} c) = (x \otimes 1) \otimes c$$
    is an isomorphism of complexes of $\CC[G]$-modules
  where the action of $G$ on the twisted complex
  $(V_\rho({\bm{t}}) \otimes_\CC \CC[G]) \otimes_{\ZZ[\pi_1({M})]}
  C_*(\ucover{M}; \ZZ)$ is given by 
  $$g \cdot ( x \otimes g' \otimes c) = x
  \otimes gg' \otimes c.$$
  and the right action of $\gamma \in \pi_1(M)$ on $V_\rho({\bm{t}}) \otimes_\CC \CC[G]$ is
  defined by 
  \[
  ((f \otimes v) \otimes g) \cdot \gamma = f \varphi(\gamma) \otimes \rho^{-1}(\gamma)(v) \otimes \pi(\gamma) g,
  \]
  where $f \in \CC(t)$, $v \in V$ and $g \in G$.
\end{lemma}
\begin{proof}[Proof of Lemma~\ref{lemma:Fox}]
  We first observe that $\Phi$ is a well--defined chain map of
  $\CC$--vector spaces since $\pi_1(\abcover{M})$ is a normal subgroup of
  $\pi_1(M)$. By the definition, we can see that
  $\Phi(x \otimes_{\pi_1(\abcover{M})} \gamma c) = \Phi(x \gamma \otimes_{\pi_1(\abcover{M})} c)$
  for any $\gamma$ in $\pi_1(\abcover{M})$.
  Hence $\Phi$ is well--defined.
  From $\Phi(\bord (x \otimes_{\pi_1(\abcover{M})} c))
  = \Phi(x \otimes_{\pi_1(\abcover{M})} \bord c) = (x \otimes 1) \otimes \bord c
  = \bord( (x \otimes 1) \otimes c) = \bord( \Phi(x \otimes c))$
  it follows that $\Phi \circ \bord = \bord \circ \Phi$.
  The $G$-equivariance of $\Phi$ follows from 
  \begin{align*}
    \Phi(g \star (x \otimes_{\pi_1(\abcover{M})} c))
    &= \Phi(x \gamma^{-1}\otimes_{\pi_1(\abcover{M})} \gamma c)\\
    &= (x \gamma^{-1} \otimes 1) \otimes \gamma c\\
    &= (x \otimes g) \otimes c \\
    &= g \cdot \Phi(x \otimes_{\pi_1(\abcover{M})} c).
  \end{align*}
  We can prove that $\Phi$ is an isomorphism by taking its inverse $\Psi$
  as $\Psi((x \otimes g) \otimes c) = g \star (x \otimes c).$
\end{proof}

We mention bases of the chain complex $V_\rho({\bm{t}}) \otimes_{\ZZ[\pi_1(\abcover{M})]} C_*(\ucover{M}; \ZZ)$
before the next step.
The following basis
\begin{equation}\label{eqn:geombasis1} 
  \cbasis[\hat{c}]{*} = 
  \bigcup_{i \mygeq 0} \{x_k \otimes_{\pi_1(\abcover{M})} \gamma_g \widetilde{e}^i_{j} \,|\, 
  1 \myleq j \myleq n_i, g \in G, 1 \myleq k \myleq m\}
\end{equation}
is the geometric basis used to compute the
polynomial torsion
$\PolyTors{\abcover{M}}{\tensor{\hatvarphi}{\hatrho}}$. When we consider the bases change 
from the basis in Equation~(\ref{eqn:geombasis1}) to the 
basis in the next equation
\begin{equation}\label{eqn:basis2}
  \{g \star(x_k \otimes_{\pi_1(\abcover{M})} \widetilde{e}^i_{j}) 
  = x_k\gamma_g^{-1} \otimes_{\pi_1(\abcover{M})} \gamma_g\widetilde{e}^i_{j}
  \,|\, 1 \myleq j \myleq n_i, g \in G, 1 \myleq k \myleq m\},
\end{equation}
we can see that the action of $\gamma_g^{-1}$ arises the change in
$\PolyTors{\abcover{M}}{\tensor{\hatvarphi}{\hatrho}}$ by
multiplying its determinant powered the Euler characteristic of
$M$. Since the Euler characteristic of $M$ is zero,
the polynomial torsion $\PolyTors{\abcover{M}}{\tensor{\hatvarphi}{\hatrho}}$ can also be
computed using the basis in Equation~(\ref{eqn:basis2}). Finally
observe that $\Phi$ maps the basis in Equation~(\ref{eqn:basis2})
to the geometric basis
\begin{equation}\label{eqn:geombasis3} 
  \cbasis{*}_G =
  \bigcup_{i \mygeq 0}
  \{(x_k \otimes g) \otimes \widetilde{e}^i_{j} \,|\,  1 \myleq j \myleq n_i, g \in G, 1 \myleq k \myleq m\},
\end{equation}
thus
\[
\PolyTors{\abcover{M}}{\tensor{\hatvarphi}{\hatrho}} = 
\tau_0(\abcover{M})^m \cdot
\Tor{
  (V_\rho(\bm{t}) \otimes_\CC \CC[G]) \otimes_{\ZZ[\pi_1(M)]} C_*(\ucover{M}; \ZZ)
}{\cbasis[\hat{c}]{*}}{\emptyset}.
\]

Now, we want to compute the torsion of 
$(V_\rho(\bm{t}) \otimes_\CC \CC[G])\otimes_{\ZZ[\pi_1({M})]} C_*(\ucover{M}; \ZZ)$ 
in terms of polynomial torsions of $M$. 
To this end we use the decomposition along 
orthogonal idempotents of the group ring $\CC[G]$, 
see~\cite{serre} for details. 
Associated to $\xi \in \hat{G}$, we define:
\[
f_\xi = \frac{1}{|G|} \sum_{g \in G} \xi(g^{-1}) g \in \CC[G].
\]
The properties of $f_\xi$ are the following
\[
f_\xi^2 = 
f_\xi, \quad f_\xi f_{\xi'} = 0 \, (\hbox{ if }\,  \xi \ne \xi'),
\quad \sum_{\xi \in \hat{G}} f_\xi = 1
\] 
and
\begin{equation*}\label{eqn:action_xi}
  g \cdot f_\xi = \xi(g) f_\xi, \text{ for all } g \in G.
\end{equation*}
We have the following $\CC[G]$-modules decomposition of the group ring
as a direct sum according to its representations:
\begin{equation}\label{eqn:decomposition}
  \CC[G] = \bigoplus_{\xi \in \hat{G}} \CC[f_\xi].
\end{equation}
Here each factor is the 1-dimensional $\CC$-vector space which is
isomorphic to the $\CC[G]$-module associated to $\xi\colon G \to \CC^*$.

Following~\cite[Section~3]{Porti:2004}, corresponding to the
decomposition in Equation~(\ref{eqn:decomposition}) we have a decomposition of
complexes of $\CC[G]$-modules:
\[
(V_\rho(\bm{t}) \otimes_\CC \CC[G])\otimes_{\ZZ[\pi_1(M)]} C_*(\ucover{M}; \ZZ) = 
\bigoplus_{\xi \in \hat{G}} (V_\rho(\bm{t}) \otimes_\CC \CC[f_\xi]) \otimes_{\ZZ[\pi_1(M)]}
C_*(\ucover{M}; \ZZ).
\]

\begin{remark}
  This decomposition implies that
  $(V_\rho({\bm{t}}) \otimes_\CC \CC[G]) \otimes_{\ZZ[\pi_1({M})]} C_*(\ucover{M}; \ZZ)$ is acyclic,
  since one can see that each chain complex
  $(V_\rho(\bm{t}) \otimes_\CC \CC[f_\xi]) \otimes_{\ZZ[\pi_1(M)]} C_*(\ucover{M}; \ZZ)$
  is acyclic from our assumptions and a change of variables.
\end{remark}

The geometric basis in Equation~(\ref{eqn:geombasis3}) induces a basis
compatible with the decomposition in
Equation~(\ref{eqn:decomposition}) by replacing $\{g \,|\, g \in G\}$
by $\{f_\xi\,|\, \xi \in \hat{G}\}$. The change of bases cancels when
we compute the torsion because Euler characteristic is zero,
see~\cite[Lemma~5.2]{Porti:2004}. And thus decomposition in
Equation~(\ref{eqn:decomposition}) implies that (in the natural
geometric bases):
\begin{align}
  \PolyTors{\abcover{M}}{\tensor{\hatvarphi}{\hatrho}}
  &= \tau_0(\abcover{M})^m \cdot 
     \Tor{(V_\rho(\bm{t}) \otimes_\CC \CC[G])\otimes_{\ZZ[\pi_1(M)]} C_*(\ucover{M}; \ZZ)}{\cbasis{*}_G}{\emptyset}
     \nonumber \\
  &= \tau_0(\abcover{M})^m \cdot 
     \prod_{\xi \in \hat{G}}
     \Tor{(V_\rho(\bm{t}) \otimes_\CC \CC[f_\xi])\otimes_{\ZZ[\pi_1(M)]} C_*(\ucover{M};\ZZ)}{\cbasis{*}}{\emptyset}. \label{eqn:decomp_tor_poly}
\end{align}

Each factor in the right hand side is related to
the polynomial torsion of $M$ and its relation is given by the following claim.

\begin{lemma}\label{lemma:poly_tors}
We have:
\[
\PolyTors{M}{(\tensor{\varphi}{\rho}) \otimes \xi} = 
  \tau_0(M)^m \cdot
  \Tor{(V_\rho(\bm{t}) \otimes_\CC \CC[f_\xi]) \otimes_{\ZZ[\pi_1(M)]} C_*(\ucover{M};\ZZ)}{\cbasis{*}}{\emptyset}.
\]
\end{lemma}
\begin{proof}[Proof of Lemma~\ref{lemma:poly_tors}]
  One can observe that, as a $\ZZ[\pi_1(M)]$-module,
  $V_\rho(\bm{t}) \otimes_\CC \CC[f_\xi]$ is isomorphic to
  $V_\rho(\bm{t})$ simply by replacing the action
  $\tensor{\varphi}{\rho}$ by $(\tensor{\varphi}{\rho})\otimes \xi$. This
  proves the equality of torsions.
\end{proof}

\begin{proof}[Proof of Theorem~\ref{theorem:torsion_covering}]
  Combining Equation~$(\ref{eqn:decomp_tor_poly})$ and
  Lemma~\ref{lemma:poly_tors}, we obtain
  $$
  \PolyTors{\abcover{M}}{\tensor{\hatvarphi}{\hatrho}} =
  \tau_0(\abcover{M})^m \cdot \tau_0(M)^{m|G|}\cdot \prod_{\xi \in
    \hat{G}} \PolyTors{M}{(\tensor{\varphi}{\rho}) \otimes \xi}
  $$
  which achieves the proof of Formula~(\ref{eqn:covering_formula}).
\end{proof}
\bibliographystyle{amsalpha}
\bibliography{ref}

\newcommand{\noop}[1]{}
\providecommand{\bysame}{\leavevmode\hbox to3em{\hrulefill}\thinspace}
\providecommand{\MR}{\relax\ifhmode\unskip\space\fi MR }
\providecommand{\MRhref}[2]{%
  \href{http://www.ams.org/mathscinet-getitem?mr=#1}{#2}
}
\providecommand{\href}[2]{#2}
\begin{thebibliography}{Tur02}

\bibitem[HM07]{HirasawaMurasugi:hakone2007}
M.~Hirasawa and K.~Murasugi, \emph{On the twisted {A}lexander {P}olynomials of
  {K}nots}, Proceedings of {H}akone {S}eminar on {G}raphs and 3--manifolds
  (M.~Yamasita, ed.), vol.~23, 2007, pp.~1--14.

\bibitem[KL99]{KL}
P.~Kirk and C.~Livingston, \emph{Twisted {A}lexander {I}nvariants,
  {R}eidemeister torsion, and {C}asson-{G}ordon invariants}, Topology
  \textbf{38} (1999), 635--661.

\bibitem[Mil66]{Milnor:1966}
J.~Milnor, \emph{Whitehead torsion}, Bull. Amer. Math. Soc. \textbf{72} (1966),
  358--426.

\bibitem[Mil68]{Milnor:1968}
\bysame, \emph{Infinite cyclic coverings}, Conference on the Topology of
  Manifolds (Michigan State Univ., E. Lansing, Mich., 1967) (1968), 115--133.

\bibitem[Por04]{Porti:2004}
J.~Porti, \emph{{Mayberry--Murasugi}'s formula for links in homology
  $3$-spheres}, Proc. Amer. Math. Soc. \textbf{132} (2004), 3423--3431.

\bibitem[Ser78]{serre}
J.P. Serre, \emph{Repr\'esentations lin\'eaires des groupes finis. ({F}rench)},
  Hermann, Paris, 1978.

\bibitem[SW09]{SilverWilliams:DynamicsTwistedAlexander}
D.~Silver and S.~Williams, \emph{Dynamics of {T}wisted {A}lexander
  {I}nvariants}, Topology Appl. \textbf{156} (2009), 2795--2811.

\bibitem[Tur86]{Turaev86:torsion_in_knot_theory}
V.~Turaev, \emph{Reidemeister torsion in knot theory}, Uspekhi Mat. Nauk
  \textbf{247} (1986), 97--147.

\bibitem[Tur01]{Turaev:2000}
\bysame, \emph{Introduction to combinatorial torsions}, Lectures in
  Mathematics, Birkh{\"a}user, 2001.

\bibitem[Tur02]{Turaev:2002}
\bysame, \emph{Torsions of $3$-dimensional manifolds}, Progress in Mathematics,
  vol. 208, Birkh{\"a}user, 2002.

\end{thebibliography}

\end{document}